\documentclass[reqno,english,12pt,a4paper]{smfart}
\usepackage[english]{babel}
\usepackage[utf8, latin1]{inputenc}

\usepackage{graphicx}
\usepackage{ae,amsfonts,euscript,enumerate}

\renewcommand{\theequation}{\thesection.\arabic{equation}}

\newtheorem{thm}{Theorem}

\newtheorem{lem}{Lemma}

\newtheorem{rem}{Remark}

\numberwithin{equation}{section}
\usepackage{fullpage}
\usepackage[T1]{fontenc}
\usepackage{graphicx}
\usepackage{amsmath}
\usepackage{amsfonts}
\usepackage{amssymb}
\usepackage{a4wide}

\def\Qbar{\overline{\mathbb Q}}
\def\f{{\bf f}}
\def\g{{\bf g}}

\newcommand{\lambd}{{\boldsymbol{\lambda}}}

\def\({\left(}
\def\){\right)}
\def\[{\left[}
\def\]{\right]}

\def\K{\mathbb K}

\author{Boris Adamczewski}
\address{
Univ Lyon, Universit\'e Claude Bernard Lyon 1 \\
CNRS UMR 5208, Institut Camille Jordan\\
F-69622 Villeurbanne Cedex, France.}
\email{Boris.Adamczewski@math.cnrs.fr}

\author{Tanguy Rivoal}
\address{Institut Fourier \\ CNRS et Universit\'e Grenoble Alpes  \\CS 40700 \\38058 Grenoble cedex 9, France.}
\email{tanguy.rivoal@univ-grenoble-alpes.fr}



\title{Exceptional values of $E$-functions at algebraic points}
\date{\today}

\thanks{This research received fundings from the European Research Council (ERC) under the European Union's Horizon 2020 research and innovation programme under the Grant Agreement No 648132, as well as from the EDATE project of the LabEx Persyval-Lab (ANR-11LABX-0025-01) funded by the French program Investissement d'avenir.}

\begin{document}

\begin{abstract} $E$-functions are entire functions with algebraic Taylor coefficients satisfying certain arithmetic conditions, and which are also solutions of  linear differential equations with coefficients in $\Qbar(z)$. They were introduced by Siegel in 1929  to generalize Diophantine properties of the exponential function, and studied further by Shidlovskii in 1956. The celebrated Siegel-Shidlovskii Theorem deals with the algebraic (in)dependence of values at algebraic points of $E$-functions solutions of a differential system. However, somewhat paradoxically, 
this deep result may fail to decide whether a given $E$-fuction assumes an algebraic or  a transcendental value at some given algebraic point. Building upon Andr\'e's theory of $E$-operators, Beukers refined in 2006 the Siegel-Shidlovskii Theorem in an optimal way. In this paper, we use Beukers' work to prove the following result: there exists an algorithm which, given a transcendental $E$-function $f(z)$ as input, outputs the finite list of all exceptional algebraic points $\alpha$ such that $f(\alpha)$ is also algebraic, together with the corresponding list of values $f(\alpha)$. This result solves the problem of  deciding whether values of $E$-functions at algebraic points are transcendental.
\end{abstract}

\maketitle

\section{Introduction}

In 1929, Siegel \cite{Siegel} wrote a landmark paper in which, amongst other important results, he introduced the notion of 
$E$-function (in a slightly more general way than below) as a generalization of the exponential function. 
Let us fix an embedding of the set of algebraic numbers $\Qbar$ into $\mathbb C$ and let us denote by 
$\mathcal O$ the ring of algebraic integers.   
A power series $f(z)=\sum_{n=0}^{\infty} \frac{a_n}{n!} z^n \in \Qbar[[z]]$ is an $E$-function if the following three conditions are fulfilled. 
\begin{enumerate}
\item[(i)] The series $f(z)$ is solution of a linear differential equation with coefficients in 
$\Qbar(z)$.
\item[(ii)] There exists $C>0$ such that for any $\sigma\in \textup{Gal}(\Qbar/\mathbb Q)$ and any $n\ge 0$,  $\vert \sigma(a_n)\vert \leq C^{n+1}$.
\item[(iii)] There exists $D>0$ and a sequence of natural numbers $d_n\neq 0$, with $\vert d_n \vert \leq D^{n+1}$, such that
$d_na_m\in \mathcal{O}$ for all~$m\le n$.
\end{enumerate}
Note that (i) implies that the $a_n$'s all lie into a certain number field $\K$. Furthermore, the function $f(z)$ is transcendental over $\mathbb C(z)$ if and only if $a_n\neq 0$ for infinitely many $n$. 

Siegel proved~\cite{Siegel} a result about the Diophantine nature of the values taken by $E$-functions at algebraic points, 
which was improved by Shidlovskii in 1956, see~\cite{shid}.

\begin{thm}[Siegel-Shidlovskii, 1956]\label{thmss}
Let $Y(z)=(f_1(z), \ldots, f_n(z))^T$ be a vector of $E$-functions such that 
$Y'(z)=A(z)Y(z)$ where $A(z)\in M_n(\Qbar(z)).$
Set $T(z)\in \Qbar[z]$ such that $T(z)A(z)\in M_n(\Qbar[z])$. 
Then for any $\alpha\in \Qbar$ such that $\alpha T(\alpha)\neq 0$, 
$$
\textup{degtr}_{\Qbar}(f_1(\alpha), \ldots, f_{n}(\alpha)) =\textup{degtr}_{\Qbar(z)}(f_1(z), \ldots, f_{n}(z)).
$$
\end{thm}

\medskip

Choosing $f(z)$ to be the exponential function, one  immediately deduces the famous Hermite-Lindemann Theorem:  
the number $e^{\alpha}$ is transcendental for all non-zero algebraic number $\alpha$. 
It is thus tempting to believe that $E$-functions should take transcendental values at non-zero algebraic points. In some sense, this is the case 
but there may be a finite number of exceptions, as illustrated by the transcendental $E$-function $\mathfrak f(z):=(z-1)e^z$ 
which vanishes at $z=1$. The reason for this exceptional behaviour is that the point $z=1$ is a \emph{singular point} with respect to the underlying 
differential system
$
\mathfrak f'(z)=(\frac{z}{z-1})\mathfrak f(z)\,,
$
that is a point such that $\alpha T(\alpha)=0$ in Theorem \ref{thmss}.~\footnote{We slightly abuse the usual terminology by considering 
that zero is always a singular point, even when it is not a pole of the matrix $A(z)$. Indeed,  
any $E$-function takes an algebraic value at zero, which makes this point a singular one from our Diophantine perspective.}  It can be shown than the latter implies 
the following simple dichotomy:  a transcendental  $E$-function solution to a differential equation 
of order one (possibly inhomogeneous) takes algebraic values~\footnote{In fact, such an $E$-function necessarily vanishes at all non-zero singular points in the homogeneous case.} 
at all singular points and transcendental values at all other algebraic points. 
However, as powerful as it is, the Siegel-Shidlovskii theorem does not completely solves the question of the 
algebraicity/transcendence for the values at algebraic points of $E$-functions satisfying higher order equations.  
There are two reasons for that. First, in the case of a differential equation of order at least two,  
 the mere transcendence of the function $f_1(z)$ does not ensure that 
the number $f_1(\alpha)$ is transcendental but only thatat least  one amongst the 
 numbers $f_1(\alpha), \ldots, f_n(\alpha)$ is transcendental, assuming furthermore that $\alpha$ is a \emph{regular point}, i.e ~a point which is not singular. 
 The second difficulty arises precisely from the fact that the Siegel-Shidlovskii Theorem does not apply at  
 singular points.

The aim of this paper is to overcome these deficiencies by proving the following result.

\begin{thm} \label{thmmain} There exists an algorithm to perform the following tasks.  
Given an $E$-function $f(z)$ as input, it first says whether $f(z)$ is transcendental or not. If so, it  outputs  the finite list of algebraic numbers $\alpha$ such that $f(\alpha)$ is algebraic, together with the corresponding list of values $f(\alpha)$. 
\end{thm}

From now on, we shall call {\em exceptional} any algebraic number, $0$ included,  where a given $E$-function takes an algebraic value. 
We shall deduce our result from the work of Beukers~\cite{beukers}, where he derived from Andr\'e's theory of $E$-operators~\cite{andre} the following refinement of the 
Siegel-Shidlovskii Theorem.    
 
 \begin{thm}[Beukers, 2006]\label{thm:ns} Under the same assumptions as in Theorem \ref{thmss},  
for any homogeneous polynomial $P\in \Qbar[X_1,\ldots, X_n]$ such that 
$P(f_1( \alpha),\ldots, f_n( \alpha))=0$, there exists a polynomial $Q\in \Qbar[Z, X_1, \ldots, X_n]$, 
homogeneous in the variables $X_1,\ldots,X_n$, such that $Q(\alpha, X_1,\ldots, X_n)=P(X_1,\ldots, X_n)$ 
and
$
Q(z,f_1(z),\ldots, f_n(z))=0.
$
\end{thm}

A similar but weaker result, in which the assumption on $\alpha$ is replaced by $\alpha\in \Qbar\setminus S$  
where $S$ is an unspecified finite set, was first proved by Nesterenko and Shidlovskii~\cite{ns} in 1996. 
Another proof of Beukers' Theorem was found later by Andr\'e~\cite{andreens}, more in the spirit of the proof of Nesterenko and Shidlovskii.   
Let us mention two consequences of Beukers' lifting  results. 
The first one is explicitly stated in~\cite{ateo} but its proof is essentially due to the referee of~\cite{firi} (where it is given in a less general case): {\em 
Let $f(z)$ be an $E$-function with Taylor coefficients all in a number field $\mathbb K$. Then for any 
$\alpha \in \Qbar$, either $f(\alpha)\notin \Qbar$ or $f(\alpha)\in \mathbb K(\alpha)$.} 
The second consequence follows from \cite[Proposition 4.1]{beukers}: {\em Let $f(z)$ be a transcendental $E$-function and let $\{\alpha_1, \ldots, \alpha_s\}$ the set of exceptional non-zero algebraic numbers for it. If $s\ge 1$,  there exist some 
integers $m_1, \ldots, m_s \ge 1$, a polynomial $p\in \Qbar[z]$ of degree $\le m_1+\cdots +m_s-1$ and an $E$-function $g(z)$ transcendental over $\Qbar(z)$ such that
$$
f(z)=p(z) + \Big(\prod_{j=1}^s (z-\alpha_j)^{m_j}\Big) g(z)
$$
and for all $\alpha\in \Qbar^*$, $g(\alpha)\notin \Qbar$.}

\medskip

Finally, we mention that analogues of all the above mentioned theorems, Theorem~\ref{thmmain} included,  have been recently proved in the setting of linear Mahler equations (see~\cite{af1, af2, pph} for statements and references). On the other hand, such results are far from being true for $G$-functions, also defined and studied by Siegel~\cite{Siegel}; see the introduction of \cite{gfndio} for an historical survey.

\medskip

The proof of Theorem~\ref{thmmain} is decomposed in four steps. In Step~1, we discuss how the function $f$ is given to us as initial input of the algorithm and how to \emph{determine} an algebraic number.  
In Step~2, the algorithm computes a minimal differential equation over $\Qbar(z)$ annihilating $f$, in fact over $\K[z]$ where $\K$ is the number field generated over $\mathbb Q$ by the Taylor coefficients of $f$. In Step~3, it computes a minimal inhomogeneous differential equation over $\K[z]$  annihilating $f$, and it determines if $f$ is transcendental. If so, let $u_0$ denote the leading polynomial of this (normalized) equation: we are then ensured by Beukers' Theorem that the exceptional non-zero $\alpha$'s all lie amongst the roots of $u_0$. 
Then, in Step~4 based on the Andr\'e-Beukers theory,the algorithm determines which roots $\alpha$ of $u_0$ are indeed such that $f(\alpha)\in \Qbar$. 
We stress that there are most $\deg(u_0)$ exceptional non-zero $\alpha$'s. Furthermore, the degree and height of  $u_0$ can  be effectively bounded a priori in terms of $L$ and $\K$. The degree and height of the corresponding $f(\alpha)$, which is in fact in $\K(\alpha)$, can also be bounded a priori in terms of $L$ and $\K$. We do not provide such explicit bounds because they depend on various huge explicit bounds 
in the literature which are already far from optimal, and thus more of theoretical than of practical interest. 
We then make some comments about effectivity. In the final section, we first illustrate our strategy with three examples. Our third example provides in particular a situation  
where $f(\alpha)$ can be transcendental even if $u_0(\alpha)=0$. In fact, this should be the typical situation. 
Thus our algorithm cannot return its output right after Step 3, and Step 4 must be performed. 

\medskip

\noindent {\bf Acknowledgments.} We very warmly thank Daniel Bertrand, Alin Bostan, Dmitri Grigoriev, Patrice Philippon, Julien Roques and Jacques-Arthur Weil  
for numerous discussions on various aspects of this project.

\section{Step 1: Comments on Theorem \ref{thmmain}}\label{sec:s1}

In this section, we first clarify  the meaning  of the expression 
\emph{Given an $E$-function $f(z)$} in Theorem \ref{thmmain}.  We also  
precise in which form the exceptional algebraic numbers (and the corresponding values taken by $f$) 
are given by our algorithm.

\subsection{How to give an $E$-function?} 

Let us write $f(z):=\sum_{n=0}^\infty \frac{a_n}{n!}z^n$. To say that $f(z)$ is an $E$-function implies that it satisfies a linear differential equation with polynomial coefficients whose coefficients are algebraic numbers, or equivalently, that the sequence $(a_n)_{n\ge 0}$ satisfies a linear recurrence with polynomial coefficients with algebraic coefficients. In order to be able to uniquely determine $f(z)$ from the knowledge of such a differential equation or such a linear recurrence, one should also know the values of $a_0,a_1,\ldots,a_m$ for a sufficiently large positive integer $m$.~\footnote{One may need more terms than the order of the recurrence. For instance, the recurrence $(n-1)a_{n}=a_{n-1}$ does not enable to compute $a_1$, whatever value is given to $a_0$; we need to be given $a_0$ and $a_1$ as initial conditions. More generally, the recurrence $\sum_{n=0}^d p_j(n)a_{n-j}=0$   is readily computed from the differential equation (see~\cite[p. 504]{bvs} or \cite[proof of Lemma 2]{gfndio} for formulas): we take  $m=\max(d,g+1)$  where $g$ is the largest positive integer root of $p_0(n)$, and $m=d$ if there is no such root.   Incidentally,  $p_0$ is the indicial polynomial 
at~$0$ of the differential equation.} 
Unfortunately, there is no known algorithm so far to check from the recurrence whether  the sequence 
$(a_n)_{n\ge 0}$ does satisfy or not  the arithmetical properties (ii) and (iii) which are requested in the definition of an $E$-function. This is similar to the fact that given a linear differential operator  in $\Qbar(z)[\frac{d}{dz}]$, there is no known algorithm  to determine whether it has an $E$-function for solution, except in some 
very specific cases.  
Furthermore, it may happen that no explicit formula for the $a_n$'s is known. For these reasons, the expression \emph{given an $E$-function $f(z)$}, as in Theorem \ref{thmmain}, will mean in this paper that:
\begin{enumerate}
\item[(i)] One knows explicitly a linear differential operator $L\in \Qbar(z)[\frac{d}{dz}]$ that annihilates $f(z)$.
\item[(ii)] One knows enough coefficients of the Taylor expansion of $f(z)$ to be able to uniquely determine $f(z)$ from the knowledge of $L$ and thus to be able to compute from $L$ as many Taylor coefficients of 
$f(z)$ as needed. (See the footnote.)
\item[(iii)] An oracle guarantees that $f(z)$ is an $E$-function. 
\end{enumerate}

\medskip

Of course, when considering an $E$-function  in practice, one often knows an explicit formula for the 
coefficients $a_n$, of hypergeometric type or involving multiple sums of multinomials divided 
by a factorial for instance. This formula should show that the sequence $a_n$ satisfies the requested properties 
(ii) and (iii) of the definition of $E$-function. Moreover, to check the differential assumption (i) for $f(z)$, 
we can try to use Zeilberger's algorithm~\cite[Chapter 7]{koepf} or its generalization to multiple 
hypergeometric sums by Weigschaider~\cite{weig}: if successfull, this provides a differential operator 
$L\in \Qbar(z)[\frac{d}{dz}]$ such that $Lf(z)=0$, but which is not necessarily minimal for the degree 
in $\frac{d}{dz}$. In theory, this approach has the defect to work only for $E$-functions with Taylor 
coefficients of multiple hypergeometric type. Again, in practice, all known examples of $E$-functions 
turn out to be of this form. In fact, Siegel~\cite{sieg} asked whether any $E$-functions 
is a linear combination of product of confluent hypergeometric series; see also~\cite[p. 184]{shid}. 
In~\cite{rivroq}, building upon certain computations done by Katz in~\cite{katz}, it is proved that 
$E$-functions of order $1$ or $2$ can be expressed with Kummer confluent hypergeometric functions  
${}_1F_1[a;b;z]$. This  answers Siegel's question in the affirmative for $E$-functions of order at most $2$, 
but the higher order cases are still open. Another possibility, that belongs to the folklore, is that any 
$E$-function could be obtained as ``specialization'' of multivariate GKZ hypergeometric series.

\subsection{How to determine an algebraic number?}

The situation is similar to the previous one.  
We say that a complex algebraic number  $\beta$ is {\em determined} if one is able to provide 
the following.    
\begin{itemize}
\item [(i)] An explicit non-zero polynomial $A\in \mathbb Q[z]$ such that $A(\beta)=0$;  
in particular, this provides explicit bounds on the degree of $\beta$ over $\mathbb Q$ and its height.
\item [(ii)] A numerical approximation of $\beta$ sufficiently accurate to be able to distinguish 
 $\beta$ from all the other roots of $A(z)$.
\end{itemize}

\section{Step 2: Finding the minimal homogeneous differential equation for $f(z)$}\label{sec:s2}

We describe here an algorithm allowing to find a  non-zero minimal homogeneous linear differential equation of a 
power series $f(z)$ solution of a given homogeneous linear differential equation with coefficients in $\Qbar$ (embedded into $\mathbb C$). Minimality is defined up to a non-zero polynomial factor; from now on, we make the slight abuse of language to write ``the'' instead of ``a''. Given a differential operator $L \in \Qbar(z)[\frac{d}{dz}]$, the degree of $L$ in $\frac{d}{dz}$ is its order; its degree in $z$ is the maximum degree amongst all the numerators and denominators of the coefficients of $L$. 

In order to find a minimal operator from $L$, we assume that: 

\begin{itemize}

\item[(i)] One knows an explicit differential operator $L\in \Qbar(z)[\frac{d}{dz}]$ annihilating $f$, say 
of order $r_0$ and of degree $\delta_0$. 

\medskip

\item[(ii)] One knows enough Taylor coefficients of $f$ in order to determine it uniquely from the 
knowledge of $L$. 
\end{itemize}

Under these assumptions, the knowledge of $L$ enables one to compute as many Taylor coefficients 
as wanted.  Let us denote by $L_{min}\in \Qbar(z)[\frac{d}{dz}]$ the minimal operator annihilating $f$. 
By minimality, $L_{min}$ is a right factor of $L$.  By \cite[Theorem 1.2]{Gr90} of Grigoriev, 
it thus follows that 
$$
\deg(L_{min})\leq \delta_1 \,,
$$
where $\delta_1$ is explicit and depends on $r_0,\delta_0$.~\footnote{We use Grigoriev's notations \cite{Gr90} in this footnote. He 
showed that 
$\delta_1=\exp(M(d_1d_2 2^n)^{o(2^n)})$ is suitable, where the quantities $M, d, d_1, n$ can be explicitely computed from the knowledge of our operator $L$. Because of the exponent $o(2^n)$, the bound might seem ineffective. In fact, his proof shows that $o(2^n)$ can be replaced 
by $\binom{n}{[n/2]}$, which Grigoriev confirmed to us~\cite{grigpriv}. Hence Theorem 1.2 of \cite{Gr90} is completely explicit. For other methods to compute such a bound, see~\cite[Sec. 9]{hoeij}.} Of course, $L_{min}$ is of order $r_1\leq r_0$.  
Let us now describe an algorithm to find $L_{min}$.

\medskip

Let $1\leq r\leq r_0$ and $0\leq \delta\leq \delta_1$. Let us assume that there exist some  polynomials 
$P_0(z),\ldots,P_r(z)$ not all zero and of degrees at most $\delta$, such that 
$$
R(z):=P_0(z)f(z) +\cdots + P_r(z)f^{(r)}(z).
$$
By the multiplicity estimate of Bertrand and Beukers \cite[Theorem 1]{BB}, one has the following alternative: 
$$
\mbox{either } \quad  R \equiv 0 \quad \mbox{or } \quad \textup{ord}_{z=0}R(z)\leq (\delta+c_1)r_0+c_2r_0^2.
$$
In \cite{BCY}, the constants $c_1$ and $c_2$ 
are made explicit, and they both depend on $\delta$ and $r_0$. One can thus find an explicit natural number 
$N$ such that 
$$
R\equiv 0 \iff \mbox{ord}_{z=0}R(z) \geq N.
$$
Lemma \ref{lem: toplitz} below then provides an algorithm to decide whether there exist some  polynomials 
$P_0(z),\ldots,P_r(z)$ not all zero and of degrees at most $\delta$, such that 
$$
P_0(z)f(z) +P_1(z)f'(z)+\cdots + P_r(z)f^{(r)}(z)=0.
$$
Then one can check, for all $1\leq r\leq r_0$ and $0\leq \delta\leq \delta_1$, 
whether there exists a differential operator $L_{r,\delta}$ of order at most $r$ and degree 
at most $\delta$ annihilating $f$. The smallest $r$ with such a property will provide $L_{min}$, as wanted.

\begin{lem}\label{lem: toplitz} 
Let $\delta$ and $N$ be two non-negative integers. 
Let 
$$
g_0(z):=\sum_{n= 0}^{\infty}a_0(n)z^n, \ldots, g_r(z):=\sum_{n= 0}^{\infty}a_r(n)z^n
$$ 
be explicitly given~\footnote{ In the sense that one can compute explicitly as many of their Taylor coefficients as needed. In the lemma, one needs to know $a_j(n)$, $j=0,\ldots, r$, $n=0, \ldots, N.$} power series in $\overline{\mathbb Q}[[z]]$.  There exists an algorithm to determine whether 
there exist some polynomials 
$P_0(z),\ldots,P_r(z)$ not all zero and of degree at most $\delta$, such that the power series 
$$
P_0(z)g_0(z) +\cdots + P_r(z)g_r(z)
$$
has order at least $N$. 
\end{lem}

\begin{proof}
Set $\mathbf{g}(z):=(g_0(z), \ldots, g_r(z))^{T}$ and 
$$
\mathbf{g}(z):=\sum_{i=0}^\infty \mathbf{g}_i z^i \,
$$
the power series expansion of $\mathbf{g}(z)$. 
Associated with the power series $\mathbf{g}(z)$, one defines the following 
$(r+1)(\delta+1)\times (N+1)$ matrix: 
\begin{equation*}
{\mathcal S_N}({\bf g}) := \left(\begin{array}{cccccc} \g_0 & \g_1 & \cdots & \g_{\delta} 
& \cdots & \g_N \\ {\mathbf 0} & \g_0 & \ddots & \ddots & \ddots & \g_{N-1} \\ 
\vdots & \ddots & &&& \vdots \\ {\mathbf 0} & \cdots &  {\mathbf 0} & \g_0 & \cdots & \g_{N-\delta} \end{array}\right) \, .
\end{equation*} 
The form of the matrix  ${\mathcal S}_N(\g)$ is reminiscent to {\it Toeplitz} matrices. 
We define the left null space, or cokernel, of 
${\mathcal S}_N(\g)$ by:
$$
{\rm coker }({\mathcal S}_N(\g)) := 
\Big\{  \lambd \in \Qbar^{(r+1)(\delta+1)} \,\big\vert \, \lambd^T \, {\mathcal S}_N(\g) = 0^T\Big\}.
$$
Then there exist some polynomials 
$P_0(z),\ldots,P_r(z)$ not all zero and of degree at most $\delta$, such that 
$$
P_0(z)g_0(z) +\cdots + P_r(z)g_r(z)
$$
has order at least $N$ if, and only if, there exists a non-zero vector 
$\lambd^T$ in ${\rm coker} ({\mathcal S}_N(\g))$. This can be algorithmically determined as it is equivalent to determine whether some linear system as a non-trivial solution or not. 
\end{proof}

\begin{rem}\emph{ Minimality of a differential equation for a given $E$-function 
can be verified by various means, especially if it is of small order, and not necessarily by the very general procedure described in this section which can be rather lengthy.}
\end{rem}

\section{Step 3: Finding the minimal inhomogeneous differential equation for $f(z)$}\label{sec:s3}

Let us assume that we are given a function $f(z)$ solution of a minimal differential equation 
\begin{equation}\label{eq:1}
\sum_{j=0}^r P_j(z)f^{(j)}(z)=0, \qquad P_j(z)\in \Qbar(z) \,\mbox{ and }\, P_r(z)\equiv 1\, .
\end{equation}

We want to find a minimal relation between $1, f(z), f'(z), \ldots,$ over $\Qbar(z)$. Either \eqref{eq:1} is such a 
minimal relation, or there exists a non-trivial relation  
\begin{equation}\label{eq:2}
1+\sum_{j=0}^{s} Q_j(z)f^{(j)}(z)=0, \qquad Q_j(z)\in \Qbar (z).
\end{equation}
for some $s\le r$. In this case, we necessarily have $s=r-1$ by minimality of $r$. Indeed, if 
otherwise $s\le r-2$, we differentiate~\eqref{eq:2} and get a non-trivial relation 
\begin{equation}\label{eq:new4}
\sum_{j=0}^s \big(Q_j(z)f^{(j)}(z)\big) '=0 .
\end{equation}
which is of the form~\eqref{eq:1} but of order $s+1<r$, contradiction.

We now want to decide if there exists a relation with $s=r-1$ . We have
\begin{align*}
0&=\sum_{j=0}^{r-1} \big(Q_j(z)f^{(j)}(z)\big) '
=\sum_{j=0}^{r-1} \big(Q_j'(z)f^{(j)}(z)+Q_j(z)f^{(j+1)}(z)\big)
\\
&=Q_{r-1}(z)f^{(r)}(z)+\sum_{j=0}^{r-1} \big(Q_j'(z)f^{(j)}(z)+Q_{j-1}(z)f^{(j)}(z)\big)  \qquad (Q_{-1}(z)=0)
\\
&=\sum_{j=0}^{r-1}\big(-P_j(z)Q_{r-1}(z) + Q_j'(z)+Q_{j-1}(z)\big) f^{(j)}(z).
\end{align*}

By minimality of $r$, we must have $-P_jQ_{r-1}+ Q_j'+Q_{j-1} =0$ for all 
$j$, i.e 
$$
Q_j'=P_jQ_{r-1}-Q_{j-1}
$$ 
for 
$j=0, \ldots, r-1$, with $Q_{-1}=0$. We then obtain a differential system:
\begin{equation}\label{eq:3}
\left(\begin{matrix}
Q_0
\\
Q_1
\\
Q_2
\\
\vdots
\\
Q_{r-1}
\end{matrix} 
\right)' = 
\left(\begin{matrix}
0 &0 &\ldots &0&P_0
\\
-1 &0 &\ldots &0&P_1
\\
0 &-1 &\ldots &0&P_2
\\
\vdots&\vdots & \vdots &\vdots&\vdots 
\\
0 &0 &\ldots &-1&P_{r-1}
\end{matrix} 
\right)
\left(\begin{matrix}
Q_0
\\
Q_1
\\
Q_2
\\
\vdots
\\
Q_{r-1}
\end{matrix} 
\right)
\end{equation}
Any $(Q_0, \ldots, Q_{r-1})\in \Qbar(z)^{r}$ such that \eqref{eq:2} holds with $s=r-1$, is a solution of \eqref{eq:3}. Conversely, if we are given any   
explicit non-zero solution $(Q_0, \ldots, Q_{r-1})\in \Qbar(z)^{r}$ of the system \eqref{eq:3}, then by construction of this system, 
we obtain an explicit  relation 
of the form \eqref{eq:new4} with $s=r-1$. Hence after integration
\begin{equation}\label{eq:new5}
\sum_{j=0}^{r-1} Q_j(z)f^{(j)}(z)=c
\end{equation}
for some constant $c$ which we now have to compute. Since the Taylor coefficients of $f(z)$ and the rational functions $Q_j(z)$ are explicitly known, we can compute the constant term of the Laurent expansion at $z=0$ of the left-hand side of \eqref{eq:new5}. This determines an algebraic number equal to $c$.

It thus remains to decide whether the system  \eqref{eq:3} 
has a non-zero rational solution and, if so, to compute it.  There exist algorithms to perform this task, for instance Barkatou's algorithm~\cite{barkatou} which works over any ground field of characteristic $0$.

\section{Step 4: Applying the Andr\'e-Beukers theory} \label{sec:step4}

In this section, we complete the proof of Theorem~\ref{thmmain}. Given an $E$-function $f(z)$ together with its minimal inhomogeneous differential equation of order $s$, 
we describe an algorithm to find the  set of algebraic points where $f$ takes algebraic 
values. 

First, if $s=0$, then $f(z)\in \Qbar[z]$ and $f(z)$ takes algebraic values at all algebraic points and our algorithm stops here.

We now assume that $s\ge 1$, so that $f(z)$ is transcendental over $\mathbb C(z)$. 
From the minimal inhomogeneous differential equation of $f(z)$ of order $s\in \{r,r-1\}$, one can find some explicit polynomials, 
$u_{0}(z), u_1(z), \ldots, u_{s+1}(z)$, with $u_0\not\equiv 0 $, such that 
\begin{equation}\label{eq:new}
\left(\begin{array}{c}0 \\ f'(z) \\ \vdots \\ f^{(s)}(z) \end{array} \right)
= \left(\begin{array}{cccccc}
 0 & 0 &0& 0& \cdots & 0  
\\
0&0&1&0 &\cdots &0
\\
0&0&0&1 &\cdots &0
\\  
\frac{u_{1}(z)}{u_0(z)} & \frac{u_{2}(z)}{u_0(z)}  &  \cdots &\cdots & \cdots & \frac{u_{s+1}(z)}{u_0(z)}
\end{array} \right)
\left(\begin{array}{c}1 \\ f(z) \\ \vdots \\ f^{(s-1)}(z) \end{array} \right) \,
\end{equation}
with $1,f(z),f'(z),\ldots,f^{(s-1)}(z)$ linearly independent over $\overline{\mathbb Q}(z)$. For later use, let $B(z)$ denote the square matrix in~\eqref{eq:new}.

 Corollary 1.4 of~\cite{beukers} implies that $1,f(\alpha),\ldots,f^{(s-1)}(\alpha)$ 
are linearly independent over $\overline{\mathbb Q}$ for any non-zero algebraic numbers which 
is not a root of $u_0(z)$, because such a point is regular for the system.   
In particular, $f(\alpha)$ is transcendental for such an algebraic number~$\alpha$. 

It thus remains to decide which roots $\alpha$ of $u_0(z)$ are such that $f(\alpha)\in \Qbar$. Note that 0 is not necessarily a root of $u_0$ (as $e^z$ shows) and we have to take  it into account in~\eqref{eq:new8} below.
By Theorem 1.5 of \cite{beukers}, there exists an $(s+1)\times (s+1)$ invertible matrix $\mathcal{M}(z)$ with entries in 
$\overline{\mathbb Q}[z]$ such that  
\begin{equation}\label{eq: systeme}
\f(z):=\left( \begin{array}{ c }
     1 \\
     f(z)\\
     \vdots \\
     f^{(s-1)}(z)
  \end{array} \right) = \mathcal{M}(z)\left( \begin{array}{ c }
     e_0(z) \\
     e_1(z) \\
     \vdots \\
     e_s(z)
  \end{array} \right)  \, ,
\end{equation}
where $e_0(z),\ldots,e_s(z)$ is a vector of $E$-functions solution of a differential system with coefficients in 
$\overline{\mathbb Q}[z,1/z]$. Since the $e_j(z)$'s are $\Qbar(z)$-linearly independent,  Corollary 1.4 of~\cite{beukers} implies again that $e_0(\alpha),\ldots,e_s(\alpha)$ are 
$\Qbar$-linearly independent for any $\alpha \in \Qbar^*$. 
Thus if $f(\alpha) \in \overline{\mathbb Q}$ then there exists $\lambd=(\beta,1,0,\ldots,0)
\in \overline{\mathbb Q}^{s+1}$ such that the scalar product 
$$
0=\lambd \cdot  \f(\alpha) = \lambd \mathcal{M}(\alpha) \left( \begin{array}{ c }
     e_0(\alpha) \\
     \vdots \\
     e_s(\alpha)
  \end{array} \right)  
$$
and thus $\lambd$ belongs to ${\rm coker} (\mathcal{M}(\alpha))$. The converse is also true and we have thus proved: 
\begin{equation}\label{eq:new8}
\big\{\alpha \in \overline{\mathbb Q} : f(\alpha)\in \overline{\mathbb Q} \big\} = \big\{\alpha\in  \overline{\mathbb Q} :  u_0(\alpha)=0 \; \textup{and} \; \exists(\beta,1,0,\ldots,0) \in  {\rm coker} (\mathcal{M}(\alpha)) \big\}\cup \big\{0\big\}.
\end{equation}
Provided the matrix $\mathcal{M}(z)$ is explicitly known, any algebraic number $\alpha$ in the set on the right-hand side of \eqref{eq:new8}, as well as the corresponding value $f(\alpha)=\beta$, is determined.

In the final part of~\cite{beukers}, Beukers constructs a suitable matrix $\mathcal{M}(z)$ by an effective ``non-zero singularity removal'' procedure, which is done one singularity after the other. Starting from a singularity $\alpha\neq 0$ of $B(z)$ of order $k$ say, 
a sequence of matrices $B_{j,\alpha}(z)$ is  explicitly computed (for $j=1$, then $j=2$, etc) each with a singularity at $\alpha$ of order $k-j$; the matrix $B_{k,\alpha}(z)$ has no singularity at $\alpha$ and we repeat the same process with its other singularities if there are any. We end up with a matrix $\mathcal{M}(z)$.~\footnote{At each step, there is a degree of freedom in the construction of a certain matrix with algebraic coefficients -- called $M$ by Beukers --, and the resulting matrix $\mathcal{M}(z)$ is not necessarily unique.} The termination of the procedure is justified at a meta level by an argument from differential Galois theory involving a fundamental matrix solution of the system \eqref{eq:new} (and not only the vector solution $(1,f(z), \ldots, f^{(s)}(z))^T$) but no explicit computation of this matrix is required to run the algorithm. We also observe that it is not always necessary to compute $\mathcal{M}(z)$ for our Diophantine purposes. For instance, if  $B(z)$ has only one singularity $\alpha\neq 0$ (of order $k$), we have $B_{k,\alpha}(z)=\mathcal{M}(z)$ but the construction of one of the matrices $B_{j,\alpha}(z)$ for some $j<k$ may already determine whether $f(\alpha)\in \Qbar$ or not; this is the case for the second example in Section~\ref{sec:ex}.

\section{Some remarks}

Our input is an explicit differential operator $L=\sum_{j=0}^\delta a_j(z)(\frac{d}{dz})^j$ in $\K[z,\frac{d}{dz}]$ and an oracle provides us with an $E$-function $f(z)\in \mathbb K[[z]]$ such that $Lf(z)=0$.  The number fields $\K$ is explicit in the sense that $\K=\mathbb Q[\beta]$, for some primitive element $\beta$ which is determined, with $p$ as minimal polynomial. This enables us  to make all the computations in Steps 2 and 3 in $\mathbb Q[X]/(p(X))$, without roundings. Similarly, in Step 4, we have to work over a finite extension $\mathbb L$ of $\K$ but again we can work in $\mathbb L=\mathbb Q[\delta]$ for some determined primitive element $\delta$.

In Step 2, we compute a minimal homogeneous differential equation $L_{\min} \in \K[z,\frac{d}{dz}]$ satisfied by $f(z)$. The degree and height of its polynomials coefficients can be a priori effectively bounded in terms of $L$, $\K$ and a certain integer $N$ (equal to the number of needed Taylor coefficients of $f(z)$) which itself depends on the degree and height of the $a_j$'s.

In Step 3, we compute a minimal inhomogeneous differential equation  satisfied by $f(z)$, with coefficients in $\K[z]$. Again, the degree and height  of its polynomial coefficients can be a priori effectively bounded in terms of $L_{\min}$ and $\K$. In particular, any non-zero 
algebraic number $\alpha$ such that $f(\alpha)\in \Qbar$ is a root of the leading polynomial coefficient $u_0$. This provides a priori bounds for he  degree and height of these (potentially) exceptional $\alpha$'s in terms of $L$ and $\K$.

In Step 4, we determine which root $\alpha$ of $u_0$ is indeed such that $f(\alpha)\in \Qbar$. A study of Beukers's procedure shows that the degree and height of $f(\alpha)\in \K(\alpha)$ can be effectively a priori bounded in terms of $u_0$ and $\alpha$. 

As already mentioned, we did not try compute these bounds explicitly because they depend on various huge explicit bounds in the literature which are already far from optimal, and thus more of theoretical than of practical interest.

\section{Examples}\label{sec:ex}

In this section, we present three examples of $E$-functions for which we compute the set of exceptional algebraic values. 
Some computations were done with the help of Maple~18.

\medskip

$\bullet$ Let us first consider the transcendental $E$-function
$$
f(z)= \sum_{n=0}^\infty \frac{1}{n!}\bigg(\sum_{k=0}^n\binom{n}{k}^2\binom{n+k}{n}\bigg) z^n.
$$
We shall prove that $f(\alpha)\notin \Qbar$ for any $\alpha\in \Qbar^*$.

The function $f(z)$ is solution of the following homogeneous differential equation, which is minimal for it because it is irreducible in $\Qbar(z)[\frac{d}{dz}]$:
\begin{equation}\label{eq:new2}
y'''(z)+ \frac{3-11z}{z} y''(z)+\frac{1-22z+z^2}{z^2} y'(z)+\frac{3-z}{z^2} y(z)=0.
\end{equation}
The minimal inhomogeneous differential equation satisfied by $f(z)$ is either~\eqref{eq:new2} or is of order 2. The latter possibility happens  if and only if  the differential system
$$
Y'(z)=
\left(\begin{array}{ccc}
 0 & 0 &  \frac{3-z}{z^2}
\\
-1&0& \frac{1-22z+z^2}{z^2}
\\
0&-1& \frac{3-11z}{z}
\end{array} \right) Y(z)
$$
has a non-zero solution $Y(z)\in \Qbar(z)^3$. As there is no such rational solution,~\eqref{eq:new2} is indeed the minimal inhomogeneous differential equation satisfied by $f(z)$. In other words, 
\begin{equation}\label{eq:new77}
\left(\begin{array}{c}0 \\ f'(z) \\ f''(z)\\f'''(z) \end{array} \right)
= \left(\begin{array}{cccc}
 0 & 0 & 0 &0
\\
0&0& 1&0
\\
0&0& 0&1
\\
0& \frac{3-z}{z^2}&\frac{1-22z+z^2}{z^2}&\frac{3-11z}{z}
\end{array} \right)
\left(\begin{array}{c}1 \\ f(z) \\ f'(z)\\f''(z) \end{array} \right).
\end{equation}
Since $0$ is the only singularity of the matrix in \eqref{eq:new77}, we deduce that $f(\alpha)\notin \Qbar$ for any $\alpha\in \Qbar^*$. Moreover, Beukers' matrix $\mathcal{M}(z)$ and basis $(e_0,e_1,e_2,e_3)$ can simply be taken as the identity matrix and $(1,f,f',f'')$ respectively, because there is no non-zero singularity to remove in \eqref{eq:new77}.

\medskip

$\bullet$ Let us now consider the transcendental $E$-function
$$
f(z)=\sum_{n=0}^\infty \frac{n^2\binom{2n}{n}}{(n+1)^2}  \frac{(z/2)^{n+1}}{n!}. 
$$
We shall prove that $f(\alpha)\notin \Qbar$ for any $\alpha\in \Qbar\setminus \{0,1\}$, and that $f(1)=\frac12$.

The function $f(z)$ is solution of the following homogeneous differential equation, which is minimal for it: 
\begin{equation}\label{eq:new3}
y'''(z)+ \frac{1-2z-2z^2}{z(1+z)} y''(z)-\frac{1+4z+z^2}{z^2(1+z)} y'(z)=0.
\end{equation}
The minimal inhomogeneous differential equation satisfied by $f(z)$ is either~\eqref{eq:new3} or is of order 2. The latter possibility happens if and only if  the differential system
$$
Y'(z)=
\left(\begin{array}{ccc}
 0 & 0 & 0
\\
-1&0& -\frac{1+4z+z^2}{z^2(1+z)}
\\
0&-1& \frac{1-2z-2z^2}{z(1+z)}
\end{array} \right) Y(z)
$$
has a non-zero solution $Y(z)\in \Qbar(z)^3$. We find that 
$$
Y(z)= \left( 1, \frac{(1-z)(1-z+2z^2)}{z(1+z)}, \frac{(1-z)^2}{1+z} \right)^T
$$
is indeed such a solution. Hence the minimal inhomogeneous differential equation satisfied by $f(z)$ is 
\begin{equation*}
y(z)+\frac{(1-z)(1-z+2z^2)}{z(1+z)}y'(z)+\frac{(1-z)^2}{1+z} y''(z) =c
\end{equation*}
for some constant $c$. Now, the constant term of the Laurent expansion at $z=0$ of 
$$
f(z)+\frac{(1-z)(1-z+2z^2)}{z(1+z)}f'(z)+\frac{(1-z)^2}{1+z} f''(z)
$$
is readily computed and seen to be equal to $\frac12$, which is our constant $c$. 
Therefore,  
\begin{equation}\label{eq:new6}
z(1-z)^2f''(z)=(z-1)(1-z+2z^2)f'(z)-z(1+z)f(z)+\frac12 z(1+z)
\end{equation}
or equivalently
\begin{equation}\label{eq:new7}
\left(\begin{array}{c}0 \\ f'(z) \\ f''(z) \end{array} \right)
= \left(\begin{array}{ccc}
 0 & 0 & 0 
\\
0&0& 1
\\
\frac{1+z}{2(z-1)^2}& -\frac{1+z}{2(z-1)^2}&\frac{1-z+2z^2}{z(z-1)}
\end{array} \right)
\left(\begin{array}{c}1 \\ f(z) \\ f'(z) \end{array} \right).
\end{equation}
At this stage, we are ensured that $f(\alpha)\notin \Qbar$ for any $\alpha \in \Qbar \setminus \{0,1\}$. To determine the arithmetic nature of $f(1)$, we start Beukers' removal process of the singularity 1 in the matrix in \eqref{eq:new7}. To do that, the first step is to multiply both sides of \eqref{eq:new7} by $(z-1)^2$ and then put $z=1$ to get a non-trivial $\Qbar$-linear relation between $1, f(1)$ and $f'(1)$. This amounts to put 
$z=1$ in \eqref{eq:new6} and we observe that this gives $f(1)=\frac12$. Hence, our problem is already solved and in fact there 
is no need to compute Beukers' matrix $\mathcal{M}(z)$. 

Finally, the Andr\'e-Beukers theory ensures that 
$f(z)=\frac{1}{2}+(z-1)g(z)$ for some $E$-function $g(z)$. It is readily checked that 
$g(z)=2\sum_{n=0}^\infty \binom{2n}{n}\frac{(z/2)^n}{n!}$. Moreover, by the same methods presented here, 
it can be proved that $g(\alpha)\notin \Qbar$ for any $\alpha \in \Qbar^*$.

\medskip

$\bullet$ Finally,  we present a class of examples showing that the roots of $u_0$ are not always exceptional values. Given two distinct integers $a,b\ge 1$, we consider the $E$-function $f(z)=z^a e^{az}+z^b e^{bz}$. The minimal differential equation satisfied by $f$ is 
\begin{equation}\label{eqdiff31}
f''(z)+\frac{1-(a+b)(1+z)^2}{z(1+z)}f'(z)+\frac{ab(1+z)^2}{z^2}f(z)=0.
\end{equation}
The latter is  easily seen to be the minimal inhomogeneous differential equation satisfied by $f$, because $a$ and $b$ are distinct. Thus $u_0(z)=z^2(1+z)$. Hence we are ensured that  $f(\alpha)\notin \Qbar$ for any $\alpha \in \Qbar \setminus \{0,-1\}$.
 However, $f(-1)=(-1)^ae^{-a}+(-1)^be^{-b}\notin \Qbar$ by the Lindemann-Weierstrass Theorem and thus there is no exceptional $\alpha\neq 0$ for $f$. Note that $-1$ is in fact exceptional for $f'(z)$ because $f'(z)=(1+z)(az^{a-1}e^{az} + bz^{b-1}e^{bz})$, so that $f'(-1)=0$; this  can be seen directly from the differential equation \eqref{eqdiff31} as well.

\end{document}